\documentclass{mcom-l}

\usepackage{latexsym,amsmath,amssymb,epsfig,bbm,bm,graphicx,xcolor,cancel}
\usepackage{mathrsfs,textcomp}
\usepackage{multirow}

\newtheorem{theorem}{Theorem}[section]
\newtheorem{lemma}[theorem]{Lemma}

\theoremstyle{definition}

\theoremstyle{remark}
\newtheorem{remark}[theorem]{Remark}

\numberwithin{equation}{section}

\newcommand{\PODr}{\sqrt{\sum\limits_{j=r+1}^d\lambda_j}}
\newcommand{\PODR}{\sqrt{\sum\limits_{j=R+1}^d\lambda_j}}

\newcommand{\AvE}{\frac{1}{N+1}{\sum\limits_{n=0}^{N}\|u^n-u_r^n\|}}

\newcommand{\MrI}{\|M_r^{-1}\|_2}

\newtheorem{corollary}{Corollary}[section]
\newtheorem{assumption}{Assumption}

\newcommand{\opartial}{\overline{\partial}}

\newcommand{\by}{{\bf y}}

\newcommand{\bx}{{\bf x}}

\def\PP{{{\rm l}\kern - .15em {\rm P} }}
\def\PN2{{\PP_{N}-\PP_{N-2}}}

\newcommand{\I}{\mathbbm{I}}

\newcommand{\R}{\mathbbm{R}}

\newcommand{\bb}{{\bf b}}
\newcommand{\fnp}{f^{n+1}}
\newcommand{\prp}{P_R^{'}}
\newcommand{\pr}{P_R}
\newcommand{\rn}{r^n}
\newcommand{\ur}{u_r}
\newcommand{\urn}{u_r^n}
\newcommand{\urnp}{u_r^{n+1}}
\newcommand{\un}{u^n}
\newcommand{\unp}{u^{n+1}}
\newcommand{\vr}{v_r}
\newcommand{\wrn}{w_r^n}
\newcommand{\etan}{\eta^n}
\newcommand{\etanp}{\eta^{n+1}}
\newcommand{\prn}{\phi_r^n}
\newcommand{\prnp}{\phi_r^{n+1}}

\begin{document}


\title[Variational Multiscale Proper Orthogonal Decomposition]{Variational Multiscale Proper Orthogonal Decomposition: Convection-Dominated Convection-Diffusion Equations}

\author{Traian Iliescu}
\address{Department of Mathematics, Virginia Polytechnic Institute 
and State University, 456 McBryde Hall, Blacksburg, Virginia 24061}
\email{iliescu@vt.edu}
\thanks{The first author was supported in part by NSF Grants \#DMS-0513542 and 
\#OCE-0620464 and AFOSR grant \#FA9550-08-1-0136}

\author{Zhu Wang}
\address{Department of Mathematics, Virginia Polytechnic Institute 
and State University, 407E McBryde Hall, Blacksburg, Virginia 24061}
\email{wangzhu@vt.edu}
\thanks{The second author was supported in part by NSF Grants \#DMS-0513542 and 
\#OCE-0620464 and AFOSR grant \#FA9550-08-1-0136.}


\date{November 24 1, 2010.}



\begin{abstract}
We introduce a variational multiscale closure modeling strategy for the numerical
stabilization of proper orthogonal decomposition reduced-order models of 
convection-dominated equations.
As a first step, the new model is analyzed and tested for convection-dominated
convection-diffusion equations.
The numerical analysis of the finite element discretization of the model is presented.
Numerical tests show the increased numerical accuracy over the standard  
reduced-order model and illustrate the theoretical convergence rates.
\end{abstract}

\maketitle

\section{Introduction}
\label{s_introduction}

One of the most successful dynamical systems ideas in the study
of turbulent flows has been the \textit{Proper Orthogonal Decomposition 
(POD)} \cite{HLB96,Sir87abc}. 
POD has been used to generate {\em reduced-order models (ROMs)} for the prediction and 
control of structure dominated turbulent flows 
\cite{AHLS88,bergmann2009enablers,buffoni2006low,NAMTT03,Pod01}.
The idea is straightforward:
Instead of using billions of local finite element basis functions equally distributed
in space, POD uses only a few (usually $\mathcal{O}(10)$) global basis functions
that represent the most energetic structures in the system.
Thus, the computational cost in a {\em direct numerical simulation (DNS)} of a complex flow
can be reduced by orders of magnitude when POD is employed.

Despite its widespread use (hundreds of papers being published every year), 
POD has several well-documented drawbacks.
In this report, we address one of them, namely the numerical instability of a
straightforward POD Galerkin procedure applied to a complex flow 
\cite{aubry1993preserving}.
To address this issue, we draw inspiration from the methodologies used in 
the numerical stabilization of finite element discretization of convection-dominated
flows.
Specifically, we employ the \textit{variational multiscale (VMS)} approach used 
by Layton in \cite{layton2002connection}, which adds artificial viscosity only to
the smallest resolved scales.
We also note that an approach similar to that used in \cite{layton2002connection}
was proposed by Guermond in \cite{guermond1999stabilization,guermond1999stabilisation}.

We emphasize that the VMS philosophy is particularly appropriate to the POD setting,
in which the hierarchy of small and large structures appears naturally.
Indeed, the POD modes are listed in decreasing order of their kinetic energy content.
We also note that, although a VMS-POD approach was announced in 
\cite{borggaard2008reduced,borggaard2011artificial}
and another VMS-POD approach was used in \cite{bergmann2009enablers}, 
to the authors' knowledge this is the first time that
the VMS formulation in \cite{layton2002connection}
has been applied in a POD setting.

In this report, the new VMS-POD model is analyzed and tested in the numerical
approximation of a {\em convection-dominated convection-diffusion} problem
\begin{eqnarray}
\label {cdcd}
\left\{ 
\begin{array}{ll}
u_t
- \varepsilon \, \Delta u
+ \bb \cdot \nabla u
+ g u
= f
& \qquad \text{ in }  (0, T] \times \Omega \, , \\
u(x,0) = u_0(x)
& \qquad \text{ in } \Omega \, , \\
u(x, t) = 0
& \qquad \text{ on }  (0, T]\times \partial \Omega \, , 
\end{array} 
\right.
\end{eqnarray}
where $\varepsilon \ll 1$ is the diffusion parameter, 
${\bf b}$ with $\| {\bf b} \| = \mathcal{O}(1)$ the given convective field,
$g$ the reaction coefficient,
$f$ the forcing term,
$\Omega \subset \R^2$ the computational domain,
$t \in [0,T]$, with $T$ the final time, and
$u_0(\cdot)$ the initial condition.
Without loss of generality, we assume in what follows that the boundary conditions
are homogeneous Dirichlet.
We emphasize that the new VMS-POD model targets turbulent flows described by
the {\em Navier-Stokes equations (NSE)}.
We chose the mathematical setting in \eqref{cdcd}, however, because it
is simple, yet relevant to our ultimate goal (since $\varepsilon \ll \| {\bf b} \|$).
Of course, once we fully understand the behavior of the new VMS-POD model in
this simplified setting, we will analyze and apply it in the NSE setting.

The rest of the paper is organized as follows.
In Section \ref{s_vms_pod}, we briefly describe the POD methodology 
and introduce the new VMS-POD model. 
The error analysis for the finite element discretization of the new 
model is presented in Section \ref{s_estimates}.
The new methodology is tested numerically in Section \ref{s_numerics} 
for a problem displaying shock-like phenomena.
Finally, Section \ref{s_conclusions} presents the conclusions and future 
research directions.

\section{Variational Multiscale Proper Orthogonal Decomposition}
\label{s_vms_pod}

\subsection{Proper Orthogonal Decomposition}
\label{ss_pod}

In this section, we briefly describe the POD.
For a detailed presentation, the reader is referred to \cite{HLB96,KV99,Sir87abc}.

Let $X$ be a real Hilbert space endowed with the inner product $(\cdot,\cdot)_X$, 
and $u(\cdot,t)\in X, t\in[0,T]$ be the state variable of a dynamical system. 
Given the time instances $t_1, \ldots, t_N \in [0,T]$, we consider the ensemble of 
snapshots
\begin{eqnarray}
   V := \mbox{span}\left\{ u(\cdot,t_1), \ldots, u(\cdot,t_N) \right\},
\label{snapshots}
\end{eqnarray}
with dim $V = d$. 
The POD seeks a low-dimensional basis $\{ \varphi_1, \ldots, \varphi_r \}$,
with $r \ll d$, which optimally approximates the input collection.
Specifically, the POD basis satisfies
\begin{eqnarray}
  \min \frac{1}{N} \sum_{i=1}^N \left\| u(\cdot,t_i) - 
  \sum_{j=1}^r \bigl( u(\cdot,t_i) \, , \, \varphi_j(\cdot) \bigr)_X \, \varphi_j(\cdot) 
  \right\|_X^2 \, ,
\label{pod_min}
\end{eqnarray}
subject to the conditions that $(\varphi_i,\varphi_j)_X = \delta_{ij}, \ 1 \leq i, j \leq r$.
In order to solve \eqref{pod_min}, we consider the eigenvalue problem
\begin{eqnarray}
K \, v = \lambda \, v \, ,
\label{pod_eigenvalue}
\end{eqnarray}
where $K \in \R^{N \times N}$, with 
$\displaystyle K_{ij} = \frac{1}{N} \, \bigl( u(\cdot,t_j) , u(\cdot,t_i) \bigr)_X \,$, 
is the snapshot correlation matrix,
$\lambda_1 \geq \lambda_2 \geq \ldots \geq \lambda_d >0$ are the positive eigenvalues, and
$v_k, \, k = 1, \ldots, d$ are the associated eigenvectors.
It can then be shown (see, e.g.,  \cite{HLB96,KV99}), that the solution
of \eqref{pod_min} is given by
\begin{eqnarray}
\varphi_{k}(\cdot) = \frac{1}{\sqrt{\lambda_k}} \, \sum_{j=1}^{N} (v_k)_j \, u(\cdot , t_j),
\quad 1 \leq k \leq r,
\label{pod_basis_formula}
\end{eqnarray}
where $(v_k)_j$ is the $j$-th component of the eigenvector $v_k$.
It can also be shown that the following error formula holds:
\begin{eqnarray}
\frac{1}{N} \sum_{i=1}^N \left\| u(\cdot,t_i) - 
  \sum_{j=1}^r \bigl( u(\cdot,t_i) \, , \, \varphi_j(\cdot) \bigr)_X \, \varphi_j(\cdot) 
  \right\|_X^2
= \sum_{j=r+1}^{d} \lambda_j \, .
\label{pod_error_formula}
\end{eqnarray}
In what follows, we will use the notation 
$X^r = \mbox{span}\{ \varphi_1, \varphi_2, \ldots, \varphi_r \} \, .$
Although $X$ can be any real Hilbert space, in what follows we consider 
$X := H_0^1(\Omega)$.

In the form it has been presented so far, POD seems to be only 
a data compression technique.
Indeed, equation \eqref{pod_min} simply says that the POD basis is the 
best possible approximation of order $r$ of the given data set.
In order to make POD a predictive tool, one couples the POD with the 
Galerkin procedure.
This, in turn, yields a ROM, i.e., a dynamical 
system that represents the evolution in time of the Galerkin truncation.
We now briefly present the derivation of this ROM, we highlight one of its 
main drawbacks and we propose a method to address this deficiency. 

The POD-Galerkin truncation is the approximation $u_r \in X^r$ of $u$:
\begin{eqnarray}
u_r(x, t)
:= \sum_{j=1}^{r} a_j(t) \, \varphi_j(x) .
\label{pod_g_truncation}
\end{eqnarray}
Plugging \eqref{pod_g_truncation} into \eqref{cdcd} and multiplying by test
functions in $X^r \subset X$ yields the {\em POD-Galerkin (POD-G)} model
\begin{eqnarray}
\hspace*{0.8cm} (u_{r, t} , \vr)
+ \varepsilon (\nabla \ur , \nabla \vr)
+ (\bb \cdot \nabla \ur , \vr) 
+ (g \, \ur, \vr) 
= (f , \vr)
\quad \forall \, \vr \in X^r .
\label{pod_g}
\end{eqnarray}
The main advantage of the POD-G model \eqref{pod_g} over a
straightforward {\em finite element (FE)} discretization of \eqref{cdcd} is clear - 
the computational cost of the former is dramatically lower than that of the latter.
There are, however, several well-documented disadvantages of \eqref{pod_g},
such as its numerical instability in convection-dominated flows \cite{sirisup2004spectral}.
To address this issue, we draw inspiration from the methodologies used in  
numerical stabilization of finite element discretizations of such flows.

\subsection{Variational Multiscale}
\label{ss_vms}

The \textit{Variational Multiscale (VMS)} method introduced by Hughes and his group
\cite{Hug95,HMJ00,HMOW01,HOM01} has been successful in the numerical stabilization 
of turbulent flows 
\cite{gravemeier2006consistent,gravemeier2004three,john2005finite,john2008finite_a,john2008finite_b,john2006two}.
The idea in VMS is straightforward: 
Instead of adding artificial viscosity to all resolved scales, in VMS artificial viscosity is
only added to the smallest resolved scales.
Thus, the small scale oscillations are eliminated without polluting the large scale 
components of the approximation.
The VMS method has been extensively developed, various numerical methods 
being used.
The finite element discretization of the resulting VMS model has evolved in several 
directions:
Hughes and his group proposed a VMS formulation for the NSE
in which a Smagorinsky model \cite{BIL05,Sma63} was added only to the smallest resolved scales 
\cite{Hug95,HMJ00,HMOW01,HOM01}.
A different type of VMS approach, based on the residual of the NSE, was proposed 
in Bazilevs et al. \cite{bazilevs2007variational}.
One of the earliest VMS ideas for convection-dominated convection-diffusion equations 
was proposed by Guermond in \cite{guermond1999stabilization,guermond1999stabilisation}.
In this VMS formulation, the smallest scales were modeled by using finite element 
spaces enriched with bubble functions.
Layton proposed in \cite{layton2002connection} a VMS approach similar to that of Guermond.
In this VMS approach, however, the smallest resolved scales were modeled by projection
on a coarser mesh.
The VMS approach proposed in \cite{layton2002connection} was extended to the NSE
in a sequence of papers by John and Kaya 
\cite{john2005finite,john2008finite_a,john2008finite_b,john2006two}.
The variational formulation used by the FE methodology fits very well with the VMS 
approach.
The definition of the smallest resolved scales, however, often poses many challenges
to the FE method.
Indeed, one needs to enrich the FE spaces with bubble functions \cite{guermond1999stabilization,guermond1999stabilisation},
consider hierarchical FE bases \cite{HMJ00}, or use a projection on a coarser mesh
\cite{layton2002connection}.

\subsection{The VMS-POD Model}
\label{ss_model}

POD represents the perfect setting for the VMS methodology, since the hierarchy of 
the basis is already present.
Indeed, the POD basis functions are already listed in descending order of their kinetic
energy content.
Based on this observation, we next propose a VMS based POD model.
To this end, we consider the following spaces:
$
X := H_0^1(\Omega), \,
X^h \subset H_0^1(\Omega), \,
X^r := \text{ span} \{ \varphi_1, \varphi_2, \ldots, \varphi_r \}, \,
X^R := \text{ span} \{ \varphi_1, \varphi_2, \ldots, \varphi_R \}, \, \text{ where } R < r,
\, \text {and } \, 
L^R,
$
where $L^R$ will be defined later.
Note that 
$
X^R \subseteq X^r \subset X^h \subset X.
$
We also consider $\pr : L^2(\Omega) \longrightarrow L^R$, the orthogonal projection of 
$L^2(\Omega)$ on $L^R$, defined by
\begin{eqnarray}
( u - \pr u , v_R ) = 0
\qquad \forall \, v_R \in L^R .
\label{def_pr}
\end{eqnarray}
Let also $\prp := \I - \pr$.
We are now ready to define the \textit{Variational Multiscale Proper Orthogonal Decomposition 
(VMS-POD) model}:
\begin{eqnarray}
\begin{split}
(u_{r, t} , \vr)
& + \varepsilon \, (\nabla \ur , \nabla \vr)
+ \alpha \, (\prp \nabla \ur , \prp \nabla \vr) \\
&  + (\bb \cdot \nabla \ur , \vr) 
+ (g \ur, \vr)
= (f , \vr)
\quad \forall \, \vr \in X^r .
\label{vms_pod}
\end{split}
\end{eqnarray}
The third term on the LHS of \eqref{vms_pod} represents the artificial viscosity
that is added only to the smallest resolved scales of the gradient.
We note that, although a VMS-POD approach was announced in 
\cite{borggaard2008reduced,borggaard2011artificial}
and another one was used in \cite{bergmann2009enablers}, to the authors' knowledge 
this is the first time that
the VMS formulation in \cite{layton2002connection} is applied in a POD setting.

In the next two sections, we will first estimate the error made in the finite element 
discretization of the new VMS-POD model \eqref{vms_pod} and then use
it in a numerical test.

\section{Error Estimates}
\label{s_estimates}

In this section, we prove estimates for the average error
$\displaystyle \frac{1}{N+1} \, \sum_{n=0}^{N} \| u^n - \urn \|$,
where the approximation $u^n$ is the solution of \eqref{cdcd_weak} 
(the weak form of \eqref{cdcd})
and 
$\urn$ is the solution of \eqref{vms_pod_fe} 
(the FE discretization of the VMS-POD model \eqref{vms_pod}). 
To this end, we follow the approach in \cite{heitmann2007subgridscale} 
(see also \cite{kaya2004numerical}).
We emphasize, however, that our presentation is different in that it has to 
include several results pertaining to the POD setting.
To this end, we use some of the developments in \cite{luo2009finite}
(see also \cite{HPS05,KV01,LCNY08,schu2010reduced}).

We start by introducing some notation and we list several results that will 
be used throughout this section.
For clarity of notation, we will denote by $C$ a generic
constant that can depend on all the parameters in the system, except on 
$d$ (the number of POD modes retained in the Galerkin truncation), 
$N$ (the number of snapshots),
$r$ (the number of POD modes used in the POD-G model \eqref{pod_g}),
$R$ (the number of POD modes used in the projection operator in the 
VMS-POD model \eqref{vms_pod}),
$h$ (the mesh-size in the FE discretization), 
$\alpha$ (the artificial viscosity coefficient), and 
$\varepsilon$ (the diffusion coefficient).
Of particular interest is the independence of the generic constant $C$ from 
$\varepsilon$.
Indeed, we will prove estimates that are {\em uniform with respect to $\varepsilon$},
which is important when convection-dominated flows (such as the NSE) are considered.

We introduce the bilinear forms
$
b(u,v)
:= (\bb \cdot \nabla u , v)
+ (g \,u , v) , \  
a(u, v)
:= \varepsilon (\nabla u, \nabla v)
+ b(u,v) , \ 
$ 
and
$
A(u,v)
:= a(u, v)
+ \alpha \, (\prp \nabla u, \prp \nabla v) .
$
We also consider the weighted norm 
$\| u \|_{a, b, \alpha}^2
:= a \, \| u \|^2 
+ b \, \| \nabla u \|^2
+ \alpha \, \| \prp \nabla u \|^2 .
$
We now make the following assumption, which is used in proving the well-posedness
of the weak formulation of \eqref{cdcd}.
\begin{assumption}[Coercivity and Continuity]
\begin{eqnarray}
g 
- \frac{1}{2} \, \nabla \cdot \bb 
\geq \beta
> 0 
\quad \text{and} \quad
\max \{ \| g \| , \| \bb \| \} 
= \gamma > 0 .
\end{eqnarray}
\label{assumption_coercivity}
\end{assumption}
For the FE discretization of \eqref{cdcd}, we consider a family 
of finite dimensional subspaces $X^h$ of $X = H_0^1(\Omega)$, such that,
for all $v \in H^{m+1} \cap X$, the following assumption is satisfied.
\begin{assumption}[Approximability]
\begin{eqnarray}
\hspace*{0.6cm} \inf_{v_h \in X^h} 
\left\{ 
\| v - v_h \|
+ h \, \| \nabla v - \nabla v^h \|
\right\}
\leq C \, h^{m+1} \, \| v \|_{m+1}
\qquad
1 \leq m \leq k ,
\label{assumption_approximability_1}
\end{eqnarray}
\label{assumption_approximability}
\end{assumption}
\noindent where $k$ is the order of accuracy of $\{ X^h \}$.
We also assume that the finite element spaces $\{ X^h \}$ satisfy the
following inverse estimate.
\begin{assumption}[FE Inverse Estimate]
\begin{eqnarray}
&& \| \nabla v_{h} \|
\leq C \, h^{-1} \ \| v_{h} \|
\qquad \forall \, v_h \in X^{h} .
\end{eqnarray}
\label{assumption_inverse_fem}
\end{assumption}
A similar inverse estimate for POD is proven in \cite{KV01}.
For completeness, we present it below.
We also include a new estimate and present its proof.
\begin{lemma}[POD Inverse Estimate]
Let $M_{r} \in \R^{r \times r}$ with $M_{i j} = (\varphi_j , \varphi_i)$ be the POD mass matrix, 
$H_{r} \in \R^{r \times r}$ with $H_{i j} = (\nabla \varphi_j , \nabla \varphi_i)$ be the POD stiffness matrix, 
$S_{r} \in \R^{r \times r}$ with $S_{i j} = (\varphi_j , \varphi_i)_{H^1}$ be the POD mass matrix in the $H^1$-norm,
and  $\| \cdot \|_2$  denote the matrix 2-norm.
Then, for all $\vr \in X^r$, the following estimates hold.
\begin{eqnarray}
\| \vr \|_{L^2} 
&\leq& \sqrt{\| M_r \|_2 \, \| S_r^{-1} \|_2 } \, \| \vr \|_{H^1} \, ,
\label{lemma_inverse_pod_1} \\
\| \vr \|_{H^1} 
&\leq& \sqrt{\| S_r \|_2 \, \| M_r^{-1} \|_2 } \, \| \vr \|_{L^2} \, ,
\label{lemma_inverse_pod_2} \\
\| \nabla \vr \|_{L^2}  
&\leq& \sqrt{\| H_r \|_2 \, \| M_r^{-1} \|_2 } \, \| \vr \|_{L^2}  \, .
\label{lemma_inverse_pod_3}
\end{eqnarray}
\label{lemma_inverse_pod}
\end{lemma}
\begin{proof}
The proof of estimates \eqref{lemma_inverse_pod_1} and \eqref{lemma_inverse_pod_2} 
was given in \cite{KV01} (see Lemma 2 and Remark 2).
The proof of \eqref{lemma_inverse_pod_3} follows along the same lines:
Let $v_r = \sum_{i=1}^{r} x_j \varphi_j$ and $\bx = (x_1, \ldots, x_r)^T$.
From the definition of $H_r$, it follows that $\| \nabla \vr \|_{L^2}^2 = \bx^T \, H_r \, \bx$.
Since $H_r$ is symmetric, its matrix 2-norm is equal to its Rayleigh quotient \cite{demmel1997applied}:
$\| H_r \|_2 = \max\limits_{\bx \neq {\bf 0}} \frac{\bx^T \, H_r \, \bx}{\bx^T \, \bx}$. 
Thus, we get:
\begin{eqnarray}
\| \nabla \vr \|_{L^2}^2 = \bx^T \, H_r \, \bx \leq \| H_r \|_2 \, \bx^T \, \bx \, .
\label{lemma_inverse_pod_3_1}
\end{eqnarray}
Furthermore, since $M_r^{-1}$ is also symmetric, we get 
$\by^T \, M_r^{-1} \, \by \leq \| M_r^{-1} \|_2 \, \by^T \, \by$ for all vectors $\by \in \R^r$.
We also note that, since $M_r$ is symmetric positive definite, we can use its Cholesky 
decomposition $M_r = L_r \, L_r^T$, where $L_r$ is a lower triangular nonsingular 
matrix \cite{demmel1997applied}.
Thus, letting $\by = L_r \, \bx$, we get:
\begin{eqnarray}
\| M_r^{-1} \|_2
\geq \frac{\by^T \, M_r^{-1} \, \by}{\by^T \, \by}
= \frac{\bx^T \, L_r^T \, (L_r^{-1})^T \, L_r^{-1} \, L_r \, \bx}{\bx^T \, L^T \, L \, \bx}
= \frac{\bx^T \, \bx}{\bx^T \, M_r \, \bx} \, .
\label{lemma_inverse_pod_3_2}
\end{eqnarray}
Inequalities \eqref{lemma_inverse_pod_3_1} and \eqref{lemma_inverse_pod_3_2} 
imply the following inequality, which proves \eqref{lemma_inverse_pod_3}:
$
\| \nabla \vr \|_{L^2}^2 
\leq \| H_r \|_2 \, \| M_r^{-1} \|_2 \, \bx^T \, M_r \, \bx
= \| H_r \|_2 \, \| M_r^{-1} \|_2 \, \| \vr \|^2_{L^2}  \, .
$
\end{proof}
\begin{remark}
We note that, in our setting, \eqref{lemma_inverse_pod_3} can be improved.
Indeed, since $S_r$ is the identity matrix when $X = H_0^1$, we get:
\begin{eqnarray}
\| \nabla \vr \|_{L^2}
\leq \| \vr \|_{H^1} 
\leq \sqrt{\| S_r \|_2 \, \| M_r^{-1} \|_2 } \, \| \vr \|_{L^2}
= \sqrt{\| M_r^{-1} \|_2 } \, \| \vr \|_{L^2} \, .
\label{remark_inverse_pod_1}
\end{eqnarray}
We note, however, that in general \eqref{remark_inverse_pod_1} might not hold.
\label{remark_inverse_pod}
\end{remark}
To prove optimal error estimates in time, we follow \cite{KV01} and include the finite difference
quotients $\opartial u(t_n) = \frac{u(t_n) - u(t_{n-1})}{\Delta t}$, where $n = 1, \ldots, N$, in the set 
of snapshots $V := \mbox{span}\left\{ u(t_0), \ldots, u(t_N), \opartial u(t_1), \ldots, \opartial u(t_N) \right\}$.
As pointed out in \cite{KV01}, the error formula \eqref{pod_error_formula} becomes:
\begin{eqnarray}
\begin{split}
& \hspace*{0.55cm} \frac{1}{2 N + 1} \, \sum_{i=0}^N 
\left\| 
u(\cdot,t_i) - \sum_{j=1}^r \bigl( u(\cdot,t_i) \, , \, \varphi_j(\cdot) \bigr)_X \, \varphi_j(\cdot) 
\right\|_X^2 \\
& + \  \frac{1}{2 N + 1} \, \sum_{i=1}^N 
\left\| 
\opartial u(\cdot,t_i) - \sum_{j=1}^r \bigl( \opartial u(\cdot,t_i) \, , \, \varphi_j(\cdot) \bigr)_X \, \varphi_j(\cdot) 
\right\|_X^2
= \sum_{j=r+1}^{d} \lambda_j \, .
\end{split}
\label{pod_error_formula_new}
\end{eqnarray}

\medskip

After these preliminaries, we are ready to derive the error estimates.

The weak form of \eqref{cdcd} reads:
\begin{eqnarray}
(u_t, v)
+ a(u,v)
= (f ,v)
\qquad \forall \, v \in X \, .
\label{cdcd_weak}
\end{eqnarray}
The VMS-POD model for \eqref{cdcd_weak} with a backward Euler time 
discretization reads: 
Find $\urn \in X^r$ such that:
\begin{eqnarray}
\frac{1}{\Delta t} \, (\urnp -\urn, \vr)
+ A(\urnp,\vr)
= (\fnp ,\vr)
\qquad \forall \, \vr \in X^r \, .
\label{vms_pod_fe}
\end{eqnarray}
The following stability result for $\urn$ holds:
\begin{theorem}
The solution $\urn$ of \eqref{vms_pod_fe} satisfies the following bound
\begin{eqnarray}
\| \urn \|
\leq \| u_r^0 \|
+ \Delta t \, \sum_{n=0}^{N-1} \| \fnp \| .
\label{theorem_stability_1}
\end{eqnarray}
\label{theorem_stability}
\end{theorem}

\begin{proof}
Choosing $\vr := \urnp$ in \eqref{vms_pod_fe}, we get:
\begin{eqnarray}
\frac{1}{\Delta t} \, (\urnp -\urn, \urnp)
+ A(\urnp, \urnp)
= (\fnp , \urnp) \, .
\label{theorem_stability_2}
\end{eqnarray}
By applying the Cauchy-Schwarz inequality on both sides of \eqref{theorem_stability_2}
and simplifying by $\| \urnp \|$, we get:
\begin{eqnarray}
\| \urnp \|
- \| \urn \|
\leq \Delta t \, \| \fnp \| .
\label{theorem_stability_3}
\end{eqnarray}
Summing from $0$ to $N-1$ the inequality in \eqref{theorem_stability_3}, we get 
\eqref{theorem_stability_1}.
\end{proof}

In order to prove an estimate for $\| u^n - \urn \|$, 
we will first consider the {\it Ritz projection} $w_r \in X^r$ of $u \in X$:
\begin{eqnarray}
A(u - w_r , \vr)
= 0
\qquad \forall \, \vr \in X^r .
\label{ritz}
\end{eqnarray}
The existence and uniqueness of $w_r$ follow from Lax-Milgram lemma.
We now prove an estimate for $\un - \wrn$, the error in the Ritz projection.

\begin{lemma}
The Ritz projection $\wrn$ of $\un$ satisfies the following error estimate:
\begin{eqnarray}
\hspace*{1.0cm} \frac{1}{N} \, \sum_{n=1}^{N} \| \un - \wrn \|
&\leq& C \, 
\Biggl\{
\left(
1
+ {\sqrt{\|M_r^{-1}\|_2}}
+ \alpha^{-1}
\right)^{1/2}  \label{lemma_ritz_0} \\
&& \hspace*{1.0cm} \left(
h^{m+1} \, \frac{1}{N} \, \sum_{n=1}^{N} \| \un \|_{m+1}
+ \sqrt{\sum_{j = r+1}^{d} \lambda_j}
\right)  \nonumber \\
&& \hspace*{0.1cm} +  \sqrt{\varepsilon + \alpha} \, 
\left(
h^{m} \, \frac{1}{N} \, \sum_{n=1}^{N} \| \un \|_{m+1}
+ \sqrt{\sum_{j = r+1}^{d} \lambda_j}
\right)
\Biggr\}  . \nonumber
\end{eqnarray}
\label{lemma_ritz}
\end{lemma}

\begin{proof}
Setting $u := \un$ in \eqref{ritz}, we get:
\begin{eqnarray}
A(\un - \wrn , \vr)
= 0
\qquad \forall \, \vr \in X^r .
\label{lemma_ritz_1}
\end{eqnarray}
We decompose the error $\un - \wrn$ as
$
\un - \wrn
:= \left( \un - I_{h, r}(\un) \right)
- \left( \wrn - I_{h, r}(\un) \right)
= \etan - \prn ,
$
where $I_{h, r}(u^n)$ is the interpolant of $u^n$ in the space $X^r$.
By the triangle inequality, we have:
\begin{eqnarray}
\frac{1}{N} \, \sum_{n=1}^{N} \| \un - \wrn \|
\leq \frac{1}{N} \, \sum_{n=1}^{N} \| \etan \|
+ \frac{1}{N} \, \sum_{n=1}^{N} \| \prn \|.
\label{lemma_ritz_3}
\end{eqnarray}
We start by estimating $\| \etan \|$.
We note that $I_{h, r}(\un)$ consists of two parts:
We first consider $u_h^n$, the FE solution of \eqref{cdcd}, which yielded the
ensemble of snapshots $V$ defined in \eqref{snapshots}.
Then, we interpolate $u_h^n$ in $X^r$, which yields $I_{h, r}(u^n)$.
Note that this is different from \cite{heitmann2007subgridscale}, where
only the first part was present (see $(8)$ in \cite{heitmann2007subgridscale}).
\begin{eqnarray}
\begin{split}
\hspace*{1.0cm} \frac{1}{N} \, \sum_{n=1}^{N} \| \etan \|
&= \frac{1}{N} \, \sum_{n=1}^{N} \| \un - I_{h, r}(\un) \|  \\
&\leq \frac{1}{N} \, \sum_{n=1}^{N} \| \un - u_h^n \|
 + \frac{1}{N} \, \sum_{n=1}^{N} \| u_h^n - I_{h, r}(\un) \| \, .
\label{lemma_ritz_4}
\end{split}
\end{eqnarray}
Using Assumption \ref{assumption_approximability}, it is easily 
shown \cite{quarteroni1994numerical} that: 
\begin{eqnarray}
\frac{1}{N} \, \sum_{n=1}^{N} \| \un - u_h^n \|
\leq C \, h^{m+1} \, \frac{1}{N} \, \sum_{n=1}^{N} \| \un \|_{m+1 }.
\label{lemma_ritz_4b}
\end{eqnarray}
Picking $\displaystyle I_{h, r}(\un) := \sum_{j=1}^{r} (u^n , \varphi_j)_X \, \varphi_j$
in the last term on the RHS of \eqref{lemma_ritz_4} and then using 
\eqref{pod_error_formula_new}, we get:
\begin{eqnarray}
\frac{1}{N} \, \sum_{n=1}^{N} \| u_h^n - I_{h, r}(\un) \| 
\leq \sqrt{\sum_{j = r+1}^{d} \lambda_j} .
\label{lemma_ritz_4c}
\end{eqnarray}
Note that we consider that the time instances $t_n = n \, \Delta t$ in the time discretization
\eqref{vms_pod_fe} are the same as the time instances at which the snapshots were taken.
If this is not the case, one should use a Taylor series approach (see $(4.8)$ in \cite{luo2009finite}).

Plugging \eqref{lemma_ritz_4b} and \eqref{lemma_ritz_4c} in \eqref{lemma_ritz_4}, 
we get: 
\begin{eqnarray}
\frac{1}{N} \, \sum_{n=1}^{N} \| \etan \|
\leq C \, h^{m+1} \, \frac{1}{N} \, \sum_{n=1}^{N} \| \un \|_{m+1 }
+ \sqrt{\sum_{j = r+1}^{d} \lambda_j} .
\label{lemma_ritz_4cd}
\end{eqnarray}
Similarly, using that $X = H_0^1$ in \eqref{pod_error_formula_new}, we get:
\begin{eqnarray}
\frac{1}{N} \, \sum_{n=1}^{N} \| \nabla \etan \|
\leq C \, h^{m} \, \frac{1}{N} \, \sum_{n=1}^{N} \| \un \|_{m+1 }
+ \sqrt{\sum_{j = r+1}^{d} \lambda_j} .
\label{lemma_ritz_4d}
\end{eqnarray}
Equation \eqref{lemma_ritz_1} implies:
\begin{eqnarray}
A(\un - \wrn, \vr)
= A(\etan - \prn, \vr)
= 0 .
\label{lemma_ritz_5}
\end{eqnarray}
Choosing $\vr = \prn$ in \eqref{lemma_ritz_5} implies:
\begin{eqnarray}
A(\prn, \prn)
= A(\etan, \prn) .
\label{lemma_ritz_6}
\end{eqnarray}
We decompose the bilinear form $A$ into its symmetric and skew-symmetric parts:
$A := A_{s} + A_{ss}$, 
where 
$A_{s}(u, v) 
:= \alpha \, (\prp \nabla u , \prp \nabla v)
+ \varepsilon \, (\nabla u, \nabla v)
+ \left( \left( g - \frac{1}{2} \nabla \cdot \bb \right) \, u , v \right)$, and
$A_{ss}(u, v)
:=  \left( \bb \cdot \nabla u + \newline \frac{1}{2} \left( \nabla \cdot \bb \right) \, u , v \right) \, $.
Equation \eqref{lemma_ritz_6} implies:
\begin{eqnarray}
A_{s}(\prn, \prn)
+ \cancelto{0}{A_{ss}(\prn, \prn)}
= A_{s}(\etan, \prn)
+ A_{ss}(\etan, \prn) \, .
\label{lemma_ritz_8}
\end{eqnarray}
Assumption \ref{assumption_coercivity} implies that
$A_{s}(\prn, \prn) \geq C \, \| \prn \|_{1,\varepsilon,\alpha}^2$.
Thus, using the Cauchy-Schwarz and Young's inequalities, 
\eqref{lemma_ritz_8} becomes:
\begin{eqnarray}
\begin{split}
\hspace*{1.0cm} & C \, \| \prn \|_{1,\varepsilon,\alpha}^2
\leq A_{s}(\prn, \prn)^{1/2} \, A_{s}(\etan , \etan)^{1/2}
+ A_{ss}(\etan, \prn) \\
& \leq \frac{1}{2} \, A_{s}(\prn, \prn) 
+ \frac{1}{2} \, A_{s}(\etan , \etan) 
+ (\bb \etan , \nabla \prn)
+ \frac{1}{2} \, \left( \left( \nabla \cdot \bb \right) \etan , \prn \right) \, .
\label{lemma_ritz_9}
\end{split}
\end{eqnarray}
Rearranging and using Assumption \ref{assumption_coercivity}, \eqref{lemma_ritz_9} becomes:
\begin{eqnarray}
\hspace*{1.0cm} C \, \| \prn \|_{1, \varepsilon, \alpha}^2
\leq C \, 
\biggl(
|A_{s}(\etan, \etan)|
+ |(\bb \etan , \nabla \prn)|
+ |\left( \left( \nabla \cdot \bb \right) \etan , \prn \right)|
\biggr) \, .
\label{lemma_ritz_10}
\end{eqnarray}
We now estimate each term on the RHS of \eqref{lemma_ritz_10}.
\begin{eqnarray}
\begin{split}
|A_{s}(\etan, \etan)|
& = \varepsilon \, \| \nabla \etan \|^2
+ \left( \left( g - \frac{1}{2} \nabla \cdot \bb \right) \, \etan , \etan \right) \\
& + \alpha \, \| \prp \nabla \etan\|^2
\leq C \, \| \etan \|_{1, \varepsilon, \alpha}^2 \, .
\label{lemma_ritz_11}
\end{split}
\end{eqnarray}
To estimate the second term on the RHS of \eqref{lemma_ritz_10}, we first 
note that $\| \pr \| \leq 1$ (since $\pr$ is $L^2$-projection) and use the inverse 
estimate \eqref{lemma_inverse_pod_2} in Lemma \ref{lemma_inverse_pod} to
obtain:
\begin{eqnarray}
\| \pr (\nabla \prn) \|
\leq \| \nabla \prn \|
\leq \| \prn \|_{H^1}
\leq \sqrt{\| M_r^{-1} \|_2} \, \| \prn \| \, .
\label{lemma_ritz_11.5}
\end{eqnarray}
Using that $(\pr u , \prp v) = 0 \ \forall \, u, v$, the Cauchy-Schwarz
and Young's inequalities, and the inverse estimate \eqref{remark_inverse_pod_1}, 
we then get:
\begin{eqnarray}
\hspace*{1.0cm} |(\bb \etan , \nabla \prn)| 
&\leq& |( \pr (\bb \etan) , \pr (\nabla \prn) )|
+ |( \prp (\bb \etan) , \prp (\nabla \prn) )| \label{lemma_ritz_12} \\
&\leq&  \| \pr (\bb \etan) \| \, \| \pr (\nabla \prn) \|
+ \| \prp (\bb \etan) \| \, \| \prp (\nabla \prn) \| \nonumber \\
&\leq&  C \, \sqrt{\| M_r^{-1} \|_2} \, \| \pr (\bb \etan) \| \, \| \prn \|
+ \| \prp (\bb \etan) \| \, \| \prp (\nabla \prn) \| \nonumber \\
&\leq&  
\left( 
\frac{1}{\beta} \, C \, \| M_r^{-1} \|_2 \, \| \pr (\bb \etan) \|^2
+ \frac{\beta}{4} \, \| \prn \|^2
\right) \nonumber \\
&+& 
\left( 
\frac{1}{2 \, \alpha} \,  \| \prp (\bb \etan) \|^2
+ \frac{\alpha}{2} \, \| \prp (\nabla \prn) \|^2
\right). \nonumber
\end{eqnarray}
We note that this is exactly why we need the inverse estimate in 
in Lemma \ref{lemma_inverse_pod}: to absorb $\| \prn \|^2$ in the LHS 
of \eqref{lemma_ritz_10}.
If we had used $\| \nabla \prn \|^2$ instead, then we would have had to absorb it
in $\varepsilon \, \| \nabla \prn \|^2$ on the LHS, and so the RHS would have
depended on $\varepsilon$.
Finally, by using the Cauchy-Schwarz and Young's inequalities, the third
term on the RHS of \eqref{lemma_ritz_10} can be estimated as follows:
\begin{eqnarray}
|\left( \left( \nabla \cdot \bb \right) \etan , \prn \right)|
\leq C \, \| \etan \| \, \| \prn \|
\leq C \, 
\left(
\frac{1}{\beta} \, \| \etan \|^2 
+\frac{\beta}{4} \, \| \prn \|^2
\right) \, .
\label{lemma_ritz_13}
\end{eqnarray}
Collecting estimates 
\eqref{lemma_ritz_10},
\eqref{lemma_ritz_11},
\eqref{lemma_ritz_12} and
\eqref{lemma_ritz_13},
we get:
\begin{eqnarray}
\begin{split}
\| \prn \|_{1, \varepsilon, \alpha}^2
\leq C \biggl(
\| \etan \|_{1, \varepsilon, \alpha}^2
& + \frac{1}{\beta} \, \| M_r^{-1} \|_2 \, \| \pr (\bb \etan) \|^2  \\
& + \frac{1}{2 \, \alpha} \,  \| \prp (\bb \etan) \|^2
+ \frac{1}{\beta} \, \| \etan \|^2 
\biggr) \, .
\label{lemma_ritz_14}
\end{split}
\end{eqnarray}
The last term on the RHS of \eqref{lemma_ritz_14}, can be absorbed in 
$C \, \| \etan \|_{1, \varepsilon, \alpha}^2$.
Since $\| \pr \| \leq 1$ ($\pr$ is $L^2$-projection) and $\| \bb \| \leq \gamma$
(by Assumption \ref{assumption_coercivity}), we get:
\begin{eqnarray}
\frac{1}{\beta} \, \| M_r^{-1} \|_2 \, \| \pr (\bb \etan) \|^2
\leq C \, \| M_r^{-1} \|_2 \, \| \etan \|^2 \, .
\label{lemma_ritz_14_a}
\end{eqnarray}
Since $\| \pr \| \leq 1$ ($\pr$ is $L^2$-projection) and $\| \bb \| \leq \gamma$
(by Assumption \ref{assumption_coercivity}), we get:
\begin{eqnarray}
\frac{1}{2 \, \alpha} \,  \| \prp (\bb \etan) \|^2
\leq \frac{C}{\alpha} \,  \| \etan \|^2 \, .
\label{lemma_ritz_14_b}
\end{eqnarray}
Thus, using \eqref{lemma_ritz_14_a} and \eqref{lemma_ritz_14_b} in 
\eqref{lemma_ritz_14}, we get:
\begin{eqnarray}
\begin{split}
\hspace*{1.0cm} \| \prn \|_{1, \varepsilon, \alpha}^2
\leq C \biggl(
\| \etan \|^2
+ \varepsilon \, \| \nabla \etan \|^2
& + \alpha \, \| \prp (\nabla \etan) \|^2
+ C \, \| M_r^{-1} \|_2 \, \| \etan \|^2 \\ 
& + \frac{1}{2 \, \alpha} \,  \| \prp (\bb \etan) \|^2
+ \frac{C}{\alpha} \,  \| \etan \|^2 
\biggr) \, .
\label{lemma_ritz_14_c}
\end{split}
\end{eqnarray}
Since $\pr$ is $L^2$-projection, $\| \prp \| \leq 1$, and thus the second
term on the RHS of \eqref{lemma_ritz_14_c} can be bounded as follows:
$\alpha \, \| \prp (\nabla \etan) \|^2 \leq \alpha \, \| \nabla \etan \|^2$.
Summing in \eqref{lemma_ritz_14_c}, we get:
\begin{eqnarray}
\begin{split}
\hspace*{1.0cm} \frac{1}{N} \, \sum_{n=1}^{N} \| \prn \|_{1, \varepsilon, \alpha}^2
& \leq C \, 
\left(
1
+ \|M_r^{-1}\|_2
+ \alpha^{-1}
\right) \,
\frac{1}{N} \, \sum_{n=1}^{N} \| \etan \|^2 \\
& + (\varepsilon + \alpha) \, 
\frac{1}{N} \, \sum_{n=1}^{N} \| \nabla \etan \|^2 \, .
\label{lemma_ritz_14_d}
\end{split}
\end{eqnarray}
Using \eqref{lemma_ritz_4cd} and \eqref{lemma_ritz_4d} in
\eqref{lemma_ritz_14_d}, we get:
\begin{eqnarray}
\hspace*{1.2cm} \frac{1}{N} \, \sum_{n=1}^{N} \| \prn \|
&\leq& C \, 
\Biggl\{
\left(
1
+ {\|M_r^{-1}\|_2}
+ \alpha^{-1}
\right)^{1/2}  \label{lemma_ritz_14_e} \\
&& \hspace*{1.0cm} 
\left(
h^{m+1} \, \frac{1}{N} \, \sum_{n=1}^{N} \| \un \|_{m+1}
+ \sqrt{\sum_{j = r+1}^{d} \lambda_j}
\right) \nonumber \\
&& \hspace*{1.0cm}
+ \sqrt{\varepsilon + \alpha} \, 
\left(
h^{m} \, \frac{1}{N} \, \sum_{n=1}^{N} \| \un \|_{m+1}
+ \sqrt{\sum_{j = r+1}^{d} \lambda_j}
\right)
\Biggr\}  . \nonumber
\end{eqnarray}
Using \eqref{lemma_ritz_3}, \eqref{lemma_ritz_4cd}, and \eqref{lemma_ritz_14_e}, we get 
\eqref{lemma_ritz_0}.
\end{proof}
\begin{corollary}
The Ritz projection $\wrn$ of $\un$ satisfies the following error estimate:
\begin{eqnarray}
\hspace*{1.0cm} \| (\un - \wrn)_t \|
&\leq& C \, 
\Biggl\{
\left(
1
+ {\|M_r^{-1}\|_2}
+ \alpha^{-1}
\right)^{1/2}  \\
&& \hspace*{1.0cm}
\left(
h^{m+1}  \| u_t \|_{L^2(H^{m+1})}
+ \sqrt{\sum_{j = r+1}^{d} \lambda_j}
\right)  \nonumber \\
&& \hspace*{1.0cm}
+ \sqrt{\varepsilon + \alpha} \, 
\left(
h^{m} \, \| u_t \|_{L^2(H^{m+1})}
+ \sqrt{\sum_{j = r+1}^{d} \lambda_j}
\right)
\Biggr\}  \nonumber .
\label{corollary_ritz_t_0}
\end{eqnarray}
\label{corollary_ritz_t}
\end{corollary}
\begin{proof}
The proof follows along the same lines as the proof of Lemma \ref{lemma_ritz}, 
with the error $\un - \wrn$ replaced by $(\un - \wrn)_t = (\eta_n - \prn)_t$.
Note that it is exactly at this point that we use the fact that the finite difference 
quotients $\opartial u(t_n)$ are included in the set of snapshots (see also Remark 1
in \cite{KV01}).
Indeed, as in the proof of Lemma \ref{lemma_ritz}, the error $(\un - \wrn)_t$ is split
into two parts:
$\displaystyle
(\un - \wrn)_t
:= \left( \un_t - I_{h, r}(\un_t) \right)
- \left( (w_t)_r^n - I_{h, r}(\un_t) \right)
= \eta_t^n - (\phi_t)_r^n .
$
As in \eqref{lemma_ritz_4}--\eqref{lemma_ritz_4cd}, $\eta_t^n$ can be estimated as
follows.
\begin{eqnarray}
\begin{split}
\| \eta_t^n \|
& \leq \| u_t^n - u_{h , t}^n \|
+ \| u_{h , t}^n - I_{h, r}(\un_t) \| \\
& \leq C \, \left(
h^{m+1}  \| u_t \|_{m+1}
+ \sqrt{\sum_{j = r+1}^{d} \lambda_j}
\right) \, ,
\label{corollary_ritz_t_1}
\end{split}
\end{eqnarray}
where in the last inequality in \eqref{corollary_ritz_t_1} we used \eqref{pod_error_formula_new}.
\end{proof}

We are now ready to prove the main result of this section.

\begin{theorem}
Assume that
\begin{eqnarray}
L^R = \nabla X^R = \text{span} \{ \nabla \varphi_1, \ldots, \nabla \varphi_R \} \, .
\label{theorem_error_0}
\end{eqnarray}
Then the following error estimate holds:
\begin{eqnarray}
&& \hspace*{0.5cm} \frac{1}{N+1} \, \sum_{n=0}^{N} \| \un - \urn \| 
\leq C \, 
\Biggl\{
\left(
1
+ \|M_r^{-1}\|_2
+ \alpha^{-1}
\right)^{1/2} \label{theorem_error_1} \\
&& \hspace*{1.0cm} \left(
h^{m+1} \, 
\frac{1}{N} \, \sum_{n=1}^{N} \left( \| \un \|_{m+1} + \| u_t \|_{L^2(H^{m+1})} \right)
+ \sqrt{\sum_{j = r+1}^{d} \lambda_j}
\right) \nonumber \\
&& \hspace*{1.0cm}
+ \sqrt{\varepsilon + \alpha} \, 
\left(
h^{m} \, \frac{1}{N} \, \sum_{n=1}^{N} \left( \| \un \|_{m+1} + \| u_t \|_{L^2(H^{m+1})} \right)
+ \sqrt{\sum_{j = r+1}^{d} \lambda_j}
\right)
\nonumber \\
&& \hspace*{1.0cm}
+ \| u^0 - u_r^0 \|
+ \Delta t \, \| u_{tt} \|_{L^2(L^2)} \nonumber \\
&& \hspace*{1.0cm}
+ \sqrt{\alpha} \, 
\left(
h^{m} \, \frac{1}{N} \, \sum_{n=1}^{N} \| u^n \|_{m+1}
+ \sqrt{\sum_{j = R+1}^{d} \lambda_j}
\right) \, 
\Biggr\} . \nonumber
\end{eqnarray}
\label{theorem_error}
\end{theorem}

\begin{proof}
We evaluate \eqref{cdcd_weak} at $t_{n+1}$, we let $v = v_r$, and then we add and subtract
$\displaystyle \left( \frac{\unp - \un}{\Delta t}, \vr\right)$:
\begin{eqnarray}
\begin{split} 
\left( 
\unp_t
- \frac{\unp - \un}{\Delta t} , \vr 
\right)
& + \left( \frac{\unp - \un}{\Delta t}, \vr \right) \\
& + a(\unp, \vr)
= (f^{n+1},\vr)  .
\label{theorem_error_1.5}
\end{split}
\end{eqnarray}
Subtracting \eqref{vms_pod_fe} from \eqref{theorem_error_1.5}, we obtain the error equation:
\begin{eqnarray}
\begin{split}
\hspace*{1.0cm} & \left( 
\unp_t
- \frac{\unp - \un}{\Delta t} , \vr 
\right)
+ \left( 
\frac{\unp - \urnp}{\Delta t} , \vr 
\right)
- \left( 
\frac{\un - \urn}{\Delta t} , \vr 
\right) \\
& \hspace*{3.5cm} + A(\unp - \urnp , \vr)
+ (a-A)(\unp , \vr)
= 0 .
\label{theorem_error_2}
\end{split}
\end{eqnarray}
We now decompose the error as
$
\un - \urn
= \bigl( \un - \wrn \bigr)
- \bigl( \urn - \wrn \bigr)
= \etan - \prn ,
$
which, by the triangle inequality, implies:
\begin{eqnarray}
\| \un - \urn \|
\leq \| \etan \| +  
\| \prn \| .
\label{theorem_error_3b}
\end{eqnarray}
We note that $\| \etan \|$ has already been bounded in Lemma \ref{lemma_ritz}.
Thus, in order to estimate the error, we only need to estimate $\| \prn \|$.
The error equation \eqref{theorem_error_2} can be written as:
\begin{eqnarray}
\begin{split}
\hspace*{1.0cm} && \left( 
\unp_t
- \frac{\unp - \un}{\Delta t} , \vr 
\right)
+ \left( 
\frac{\etanp - \etan}{\Delta t} , \vr 
\right)
- \left( 
\frac{\prnp - \prn}{\Delta t} , \vr 
\right) \\
&& + A(\etanp - \prnp , \vr)
+ (a-A)(\unp , \vr)
= 0 .
\label{theorem_error_4}
\end{split}
\end{eqnarray}
We pick $\vr := \prnp$ in \eqref{theorem_error_4}, we note that, 
since $\prnp \in X^r$, $A(\etanp , \prnp) = 0$, and we get:
\begin{eqnarray}
\begin{split}
\hspace*{1.0cm} & A(\prnp , \prnp)
+ \frac{1}{\Delta t} \, (\prnp - \prn , \prnp)
= \frac{1}{\Delta t} \, (\etanp - \etan , \prnp) \\
& \hspace*{4.0cm} + ( \rn , \prnp ) 
+ (a-A)(\unp , \prnp) ,
\label{theorem_error_5}
\end{split}
\end{eqnarray}
where $\displaystyle \rn = \unp_t - \frac{\unp - \un}{\Delta t}$.
We now start estimating all the terms in \eqref{theorem_error_5}.
The terms on the LHS of \eqref{theorem_error_5} are estimated 
as follows:
\begin{eqnarray}
&& A(\prnp , \prnp)
\geq \beta \, \| \prnp \|^2
+ \varepsilon \, \| \nabla \prnp \|^2
+ \alpha \, \| \prp \nabla \prnp \|^2 .
\label{theorem_error_6} \\
&& \frac{1}{\Delta t} \, (\prnp - \prn , \prnp)
\geq \frac{1}{\Delta t} \, \left(\| \prnp \|^2 - \| \prn \| \, \| \prnp \| \right) .
\end{eqnarray}
Now we estimate the RHS of \eqref{theorem_error_5} 
by using the Cauchy-Schwarz and Young's inequalities:
\begin{eqnarray}
\begin{split}
\hspace*{1.0cm} & \left(
\frac{1}{\Delta t} \, 
(\etanp - \etan)
+ \rn , \prnp
\right)
\leq \left\| \frac{1}{\Delta t} \, (\etanp - \etan) + \rn\right\| \, \| \prnp \|  \\
& \hspace*{3.0cm} \leq \frac{1}{2 \, \beta} \,  \left\| \frac{1}{\Delta t} \, (\etanp - \etan) + \rn\right\|^2
+ \frac{\beta}{2} \, \| \prnp \|^2 . \label{theorem_error_7} \\
\end{split}
\end{eqnarray}
\begin{eqnarray}
\begin{split}
\hspace*{1.0cm} & (a-A)(\unp , \prnp)
= - \alpha \, (\prp \nabla \unp , \prp \nabla \prnp)  \\
& \leq \alpha \, \| \prp \nabla \unp \| \, \| \prp \nabla \prnp \|
\leq \frac{\alpha}{2} \, \| \prp \nabla \unp \|^2
+ \frac{\alpha}{2} \, \| \prp \nabla \prnp \|^2 .
\label{theorem_error_8}
\end{split}
\end{eqnarray}
Using \eqref{theorem_error_6}-\eqref{theorem_error_8} and absorbing 
RHS terms into LHS terms, \eqref{theorem_error_5} now reads: 
\begin{eqnarray}
\begin{split}
\hspace*{1.0cm} & \frac{1}{\Delta t} \, (\| \prnp \|^2 - \| \prn \| \, \| \prnp \|)
+ \frac{\beta}{2} \, \| \prnp \|^2
+ \varepsilon \, \| \nabla \prnp \|^2 \\
& + \frac{\alpha}{2} \, \| \prp \nabla \prnp \|^2 
\leq \frac{1}{2 \, \beta} \,  
\left\| \frac{1}{\Delta t} \,(\etanp - \etan) + \rn\right\|^2
+ \frac{\alpha}{2} \, \| \prp \nabla \unp \|^2 .
\label{theorem_error_9}
\end{split}
\end{eqnarray}
By using Young's inequality, the first term on the LHS of \eqref{theorem_error_9} 
can be estimated as follows:
\begin{eqnarray}
\begin{split}
\hspace*{1.0cm} \| \prnp \|^2 - \| \prn \| \, \| \prnp \|
& \geq \| \prnp \|^2 
- \frac{1}{2} \, \| \prn \|^2 
- \frac{1}{2} \, \| \prnp \|^2 \\
& = \frac{1}{2} \, \| \prnp \|^2
-\frac{1}{2} \, \| \prn \|^2 .
\label{theorem_error_10}
\end{split}
\end{eqnarray}
Using \eqref{theorem_error_10} in \eqref{theorem_error_9} and multiplying by 
$2 \, \Delta t$, we get:
\begin{eqnarray}
&&  \hspace*{1.0cm} \| \prnp \|^2
-\| \prn \|^2
+ \Delta t \, \| \prnp \|_{1, \varepsilon, \alpha}^2 
\label{theorem_error_11} \\
&& \hspace*{2.0cm} \leq C \, \left( \Delta t \, \left\| \frac{1}{\Delta t} \, (\etanp - \etan) + \rn \right\|^2
+ \alpha \, \Delta t \, \| \prp \nabla \unp \|^2 \right) \nonumber \\
&& \hspace*{2.0cm} \leq C \, 
\left( \Delta t \, \left\| \frac{1}{\Delta t} \, (\etanp - \etan) \right\|^2 + \Delta t \, \| \rn \|^2
+ \alpha \, \Delta t \, \| \prp \nabla \unp \|^2 \right) . \nonumber
\end{eqnarray}
Summing from $n=0$ to $n=N-1$ in \eqref{theorem_error_11}, we get:
\begin{eqnarray}
\begin{split} 
\hspace*{1.0cm} & \max_{0 \leq n \leq N} \| \phi^n_r \|^2 
+ \sum_{n=0}^{N-1} \Delta t \, \| \prnp \|_{1, \varepsilon, \alpha}^2 
\leq C \, \Biggl( 
\Delta t \, \sum_{n=0}^{N-1} \left\| \frac{1}{\Delta t} \, (\etanp - \etan) \right\|^2 \\
& \hspace*{2.0cm} + \| \phi^0_r \|^2  
+  \Delta t \, \sum_{n=0}^{N-1} \left\| \rn \right\|^2
+ \, \alpha \, \Delta t \, \sum_{n=0}^{N-1} \| \prp \nabla \unp \|^2 \Biggr) .
\label{theorem_error_12}
\end{split}
\end{eqnarray}
Proceeding as in \cite{thomee2006galerkin} (see also \cite{heitmann2007subgridscale}), 
we estimate the first term on the RHS of \eqref{theorem_error_12} as follows.
We start by writing:
\begin{eqnarray}
\etanp - \etan
= \int_{t_{n}}^{t_{n+1}} \eta_t \, dt \, .
\label{theorem_error_12a}
\end{eqnarray}
Taking the $L^2$-norm in \eqref{theorem_error_12a} and applying the Cauchy-Schwarz 
inequality, we get:
\begin{eqnarray}
\begin{split} 
\hspace*{1.0cm} \| \etanp - \etan \|
& \leq \int_{t_{n}}^{t_{n+1}} 1 \, \| \eta_t \| \, dt
\leq \left( \int_{t_{n}}^{t_{n+1}} 1^2 \, dt \right)^{1/2} \, 
\left( \int_{t_{n}}^{t_{n+1}} \| \eta_t \|^2 \, dt \right)^{1/2} \\
&\leq (\Delta t)^{1/2} \, \left( \int_{t_{n}}^{t_{n+1}} \| \eta_t \|^2 \, dt \right)^{1/2} \, ,
\label{theorem_error_12b}
\end{split}
\end{eqnarray}
which implies
$\displaystyle
\Delta t \, \left\| \frac{1}{\Delta t} \, (\etanp - \etan) \right\|^2
\leq \left( \int_{t_{n}}^{t_{n+1}} \| \eta_t \|^2 \, dt \right)^{1/2}
$.
Summing from $n=0$ to $n=N-1$, we get
$\displaystyle
\Delta t\, \sum_{n=0}^{N-1} \left\| \frac{1}{\Delta t} \, (\etanp - \etan) \right\|^2
\leq \| \eta_t \|_{L^2(L^2)},
$
which was bound in Corollary \ref{corollary_ritz_t}.
We thus obtain:
\begin{eqnarray}
&& \Delta t\, \sum_{n=0}^{N-1} \left\| \frac{1}{\Delta t} \, (\etanp - \etan) \right\|^2
\leq C \, 
\Biggl\{
\left(
1
+ {\|M_r^{-1}\|_2}
+ \alpha^{-1}
\right)^{1/2} \label{theorem_error_12c} \\
&& \hspace*{4.0cm} \left(
h^{m+1}  \| u_t \|_{L^2(H^{m+1})}
+ \sqrt{\sum_{j = r+1}^{d} \lambda_j}
\right)  \nonumber \\
&& \hspace*{3.0cm} + \sqrt{\varepsilon + \alpha} \, 
\left(
h^{m} \, \| u_t \|_{L^2(H^{m+1})}
+ \sqrt{\sum_{j = r+1}^{d} \lambda_j}
\right)
\Biggr\}  . \nonumber
\end{eqnarray}
To estimate the third term on the RHS of \eqref{theorem_error_12}, we use a Taylor series
expansion of $\un$ around $\unp$:
\begin{eqnarray}
\un
= \unp
- \unp_t \, \Delta t
+ \int_{t_{n}}^{t_{n+1}} u_{tt}(s) \, (t_n - s) \, ds \, .
\label{theorem_error_12d}
\end{eqnarray}
Taking the $L^2$-norm in \eqref{theorem_error_12d} and applying the Cauchy-Schwarz 
inequality, we get
$\displaystyle
\| \rn \| 
\leq \int_{t_{n}}^{t_{n+1}}  1 \, \| u_{tt} \| \, ds
\leq (\Delta t)^{1/2} \, \| u_t \|_{L^2(L^2)} \, .
$
Summing from $n=0$ to $n=N-1$, we get:
\begin{eqnarray}
\Delta t\, \sum_{n=0}^{N-1} \left\| \rn \right\|^2
\leq  \Delta t^2 \, \| u_{tt} \|^{2}_{L^2(L^2)}  \, .
\label{theorem_error_14} 
\end{eqnarray}
To estimate the last term on the RHS of \eqref{theorem_error_12}, we
use the fact that $L^R = \nabla X^R$ (assumption \eqref{theorem_error_0}).
We emphasize that this is the {\it only instance} in the proof where the assumption 
$L^R = \nabla X^R$ is used.
Thus, we get:
\begin{eqnarray}
&& \alpha \, \Delta t \, \sum_{n=0}^{N-1} \| \prp \nabla \unp \|^2 
= \alpha \, \Delta t \, \sum_{n=0}^{N-1} \| \nabla \unp  - \pr \nabla \unp\|^2 \label{theorem_error_15} \\
&& \hspace*{2.5cm}  \stackrel{\eqref{theorem_error_0}}{\leq} C \, \alpha \, \frac{1}{N} \, \sum_{n=0}^{N-1} 
\inf_{v_R \in X^R} \| \nabla \unp  - \nabla v_R \|^2
\nonumber \\
&& \hspace*{2.5cm} \stackrel{\eqref{pod_error_formula}, \eqref{assumption_approximability_1}}{\leq} 
C \, \alpha  \, 
\left(
h^{m} \, \frac{1}{N} \, \sum_{n=1}^{N} \| u^n \|_{m+1 }
+ \sqrt{\sum_{j = R+1}^{d} \lambda_j}
\right)^2 \, . \nonumber
\end{eqnarray}

Using \eqref{theorem_error_12c}, \eqref{theorem_error_14}, and \eqref{theorem_error_15} in 
\eqref{theorem_error_12}, 
the obvious inequality 
$\displaystyle \max_{0 \leq n \leq N} \| \prn \| 
\geq \frac{1}{N+1} \, \sum_{n=0}^{N} \| \prn \|$,
inequality \eqref{theorem_error_3b}, and the 
estimates in Lemma \ref{lemma_ritz}, we obtain the error estimate \eqref{theorem_error_1}.
\end{proof}

\section{Numerical Results}
\label{s_numerics}

The goal of this section is twofold:
(i) to show that the new VMS-POD model \eqref{vms_pod} is significantly more
stable numerically than the standard POD-G model \eqref{pod_g}; and
(ii) to illustrate numerically the theoretical error estimate \eqref{theorem_error_1}.
We also use Theorem \ref{theorem_error} to provide theoretical guidance in 
choosing an optimal value for the artificial viscosity coefficient $\alpha$ and use this algorithm 
within our numerical framework.
Finally, we show that the VMS-POD model \eqref{vms_pod} displays a relatively
low sensitivity with respect to changes in the diffusion coefficient $\varepsilon$.
Thus, we provide numerical support for the theoretical estimate \eqref{theorem_error_1},
which is {\em uniform} with respect to $\varepsilon$.

The mathematical model used for all the numerical tests in this section is the 
convection-dominated convection-diffusion equation \eqref{cdcd} with the following
parameter choices:
spatial domain $\Omega=[0, 1]\times[0, 1]$,
time interval $[0, T]=[0, 1]$,
diffusion coefficient $\varepsilon = 1\times 10^{-4}$,
convection field ${\bf b} = [\cos\frac{\pi}{3}, \sin\frac{\pi}{3}]^{T}$, and 
reaction coefficient $g=1$. 
The forcing term $f$ and initial condition $u_0(x)$ are chosen to satisfy the exact
solution $u(x,y,t) = 0.5\sin(\pi x)\sin(\pi y) \left[ \tanh\left(\frac{x+y-t-0.5}{0.04}\right) +1 \right]$, which is 
similar to that used in \cite{guermond1999stabilization}.
As in the theoretical developments in Section \ref{s_estimates}, in this section we
employ the finite element method for spatial discretization and the backward Euler
method for temporal discretization of all models investigated.
All computations are carried out on a PC with 3.2 GHz Intel Xeon Quad-core processor.

We start by comparing the VMS-POD model \eqref{vms_pod} to the standard POD-G
model \eqref{pod_g}.
To generate the POD basis, we first run a DNS
with the following parameters: piecewise quadratic finite elements, uniform triangular mesh 
with mesh-size $h = 0.01$, and time-step $\Delta t = 10^{-4}$.
A mesh refinement study indicates that DNS mesh resolution is achieved.
The average DNS error is 
$\frac{1}{N+1}{\sum\limits_{n=0}^{N}\|u^n-u_h^n\|} = 2.04\times10^{-4}$, where 
$N=1000$, and $u^n$ and $u^n_h$ are the exact solution and the finite element 
solution at $t=n\Delta t$, respectively.
The CPU time of the DNS is $9.42 \times 10^4\, s$.
Since the forcing term is time-dependent, the global load vectors are stored for later 
use in all the ROMs. 
The POD modes are generated in $H^1$ by the  method of snapshots; 
the rank of the data set is $104$.
For both POD-ROMs (POD-G and VMS-POD), we use the same number of POD basis
functions: $r = 40$.
\begin{table}[ht]
\centering
\caption{Average errors for the POD-G model \eqref{pod_g} with different values of $r$. 
Note that the POD-G model yields poor results.}
\label{table_PODG}
\begin{tabular}{|c|c|c|c|c|}
\hline
$r$&20&40&60&80\\
\hline
$\AvE$&$1.25\times 10^{-1}$&$1.11\times 10^{-1}$&$9.28\times 10^{-2}$&$8.20\times 10^{-2}$\\
\hline
\end{tabular}
\end{table}

We first test the POD-G model \eqref{pod_g}. 
The CPU time for the POD-G model is $96.4\, s$, which is three orders of magnitude
lower than that of a brute force DNS.
The numerical solution at $t = 1$ is shown in Figure \ref{Comp_models} for both the
DNS (top) and the POD-G model (middle).
It is clear from this figure that, although the first $40$ POD modes capture $99.99\%$
of the system's kinetic energy, the POD-G model yields poor quality results and 
displays strong numerical oscillations. 
This is confirmed by the POD-G model's high average error $\AvE =1.11\times 10^{-1}$, 
where $u_r^n$ is the POD-G model's solution at $t=n\Delta t$.
Indeed, the POD-G model's average error is almost three orders of magnitude higher
than the average error of the DNS.
The average errors for different values of $r$ listed in Table \ref{table_PODG} show that
increasing the number of POD modes ($r$) does not decrease significantly the average
error.
It is thus clear that the straightforward POD-G model, although computationally efficient,
is highly inaccurate.

Next, we investigate the VMS-POD model \eqref{vms_pod}.
We make the following parameter choices: $R=20$ and $\alpha=4.29\times 10^{-2}$.
The motivation for this choice is given later in this section.
The CPU time for the VMS-POD model \eqref{vms_pod} is $106.2\,s$, which is close 
to the CPU time of the POD-G model \eqref{pod_g}.
The numerical solution at $t = 1$ for the VMS-POD model is shown in Figure 
\ref{Comp_models} (bottom).
It is clear from this figure that the VMS-POD model is much more accurate than the
POD-G model.
Indeed, the VMS-POD model results are much closer to the DNS results than the
POD-G model results, since the numerical oscillations displayed by the latter are
dramatically decreased.
This is confirmed by the VMS-POD model's average error  $\AvE = 4.48\times 10^{-3}$, 
where $u_r^n$ is the VMS-POD solution at $t=n\Delta t$; this error is more than $20$ 
times lower than the error of the POD-G model.
\begin{figure}[h]
\centering
\caption{Numerical solution at $t = 1$:
DNS (top), POD-G model \eqref{pod_g} (middle), and VMS-POD model \eqref{vms_pod}
(bottom).
Note that the VMS-POD model is much more accurate than the POD-G model, decreasing
the unphysical oscillations of the latter.
The CPU times for both the VMS-POD and  POD-G models are three orders of magnitude 
lower than the CPU time for the DNS.\vspace*{0.1cm}
}\label{Comp_models}
\includegraphics[width=0.6\textwidth]{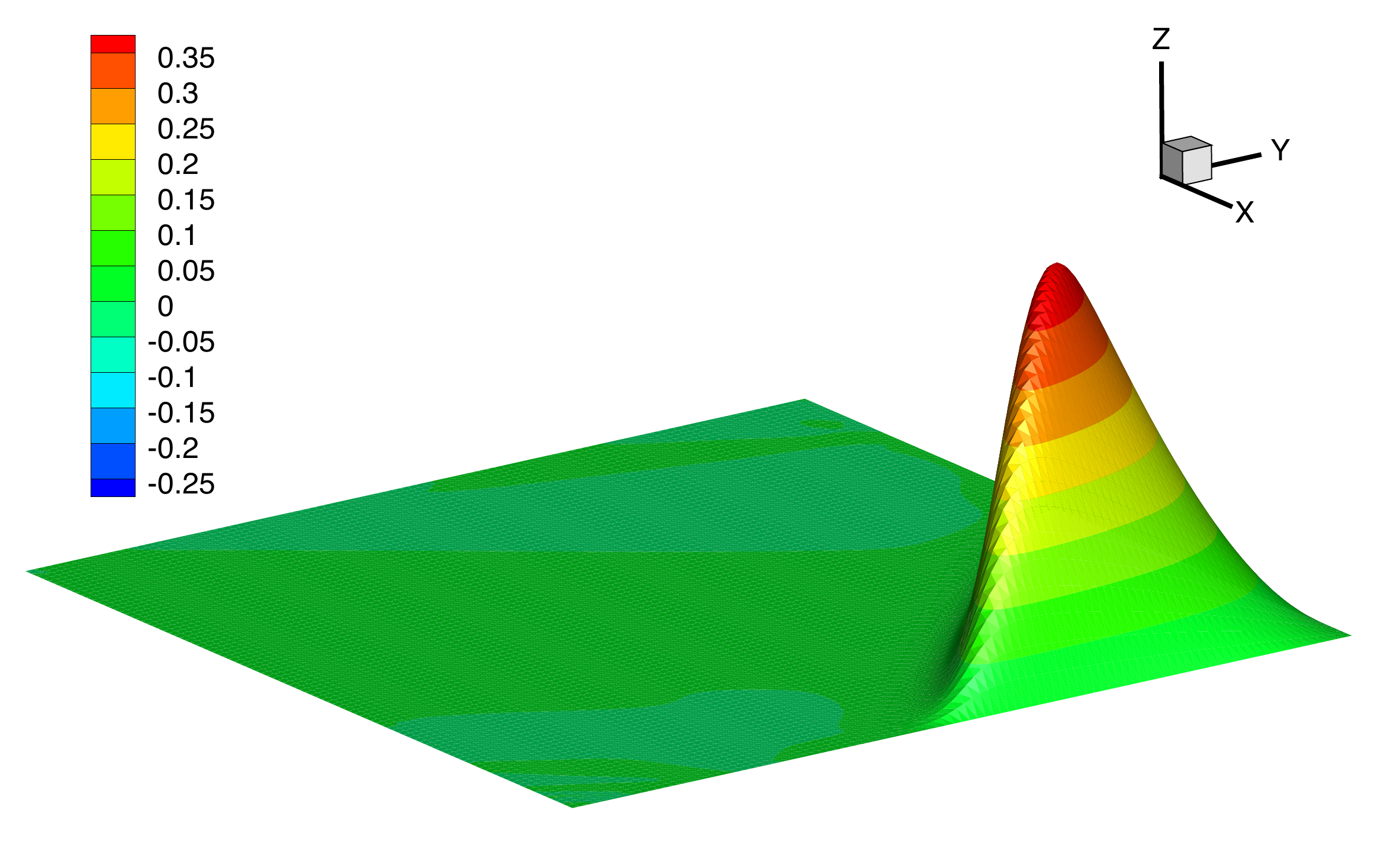}
\includegraphics[width=0.6\textwidth]{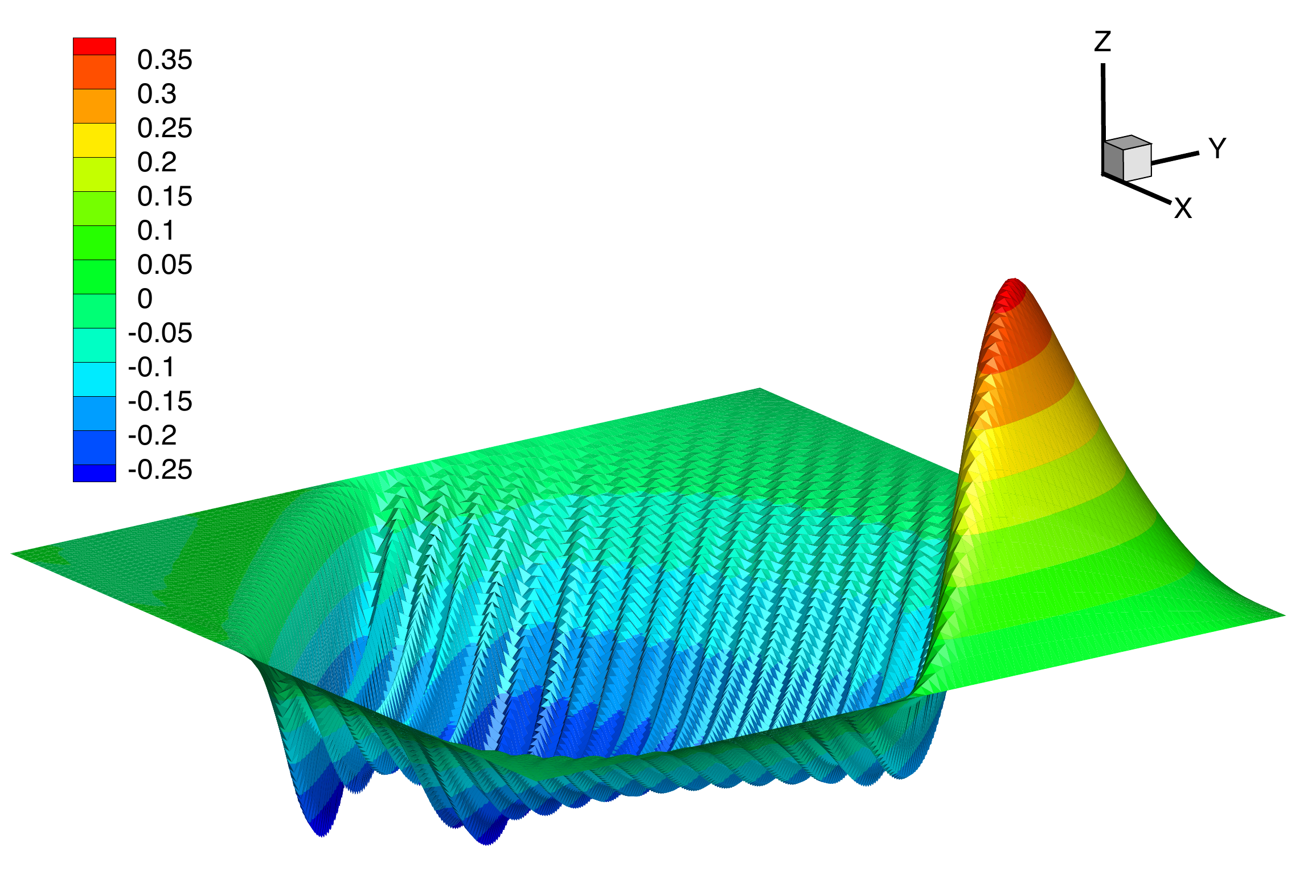}
\includegraphics[width=0.6\textwidth]{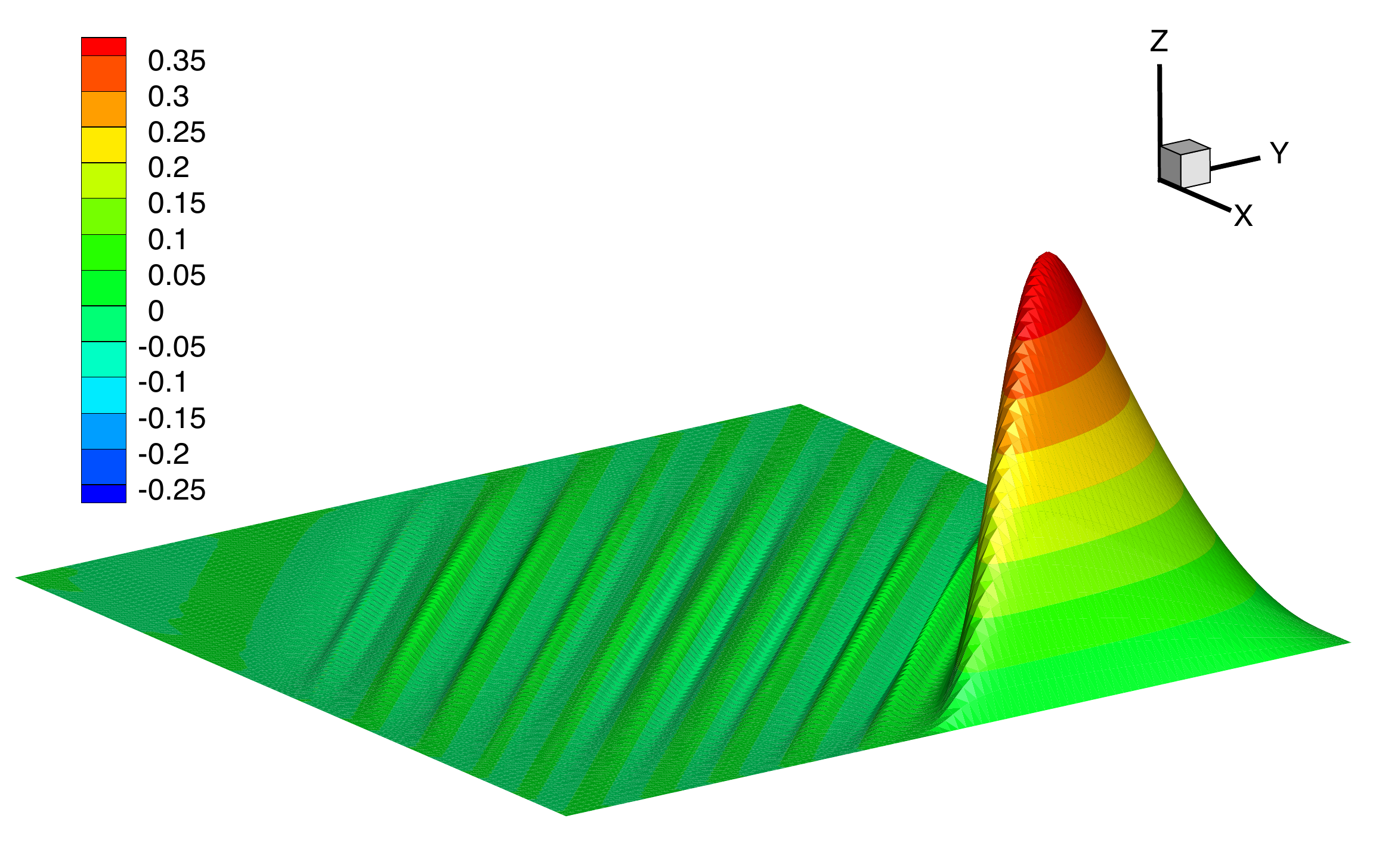}
\end{figure}
In conclusion, the VMS-POD model \eqref{vms_pod} dramatically decreases the error
of the POD-G model \eqref{pod_g} by adding numerical stabilization, while keeping the
same level of computational efficiency.
\begin{table}[ht]
\centering
\caption{VMS-POD model's average error $e = \AvE$ and its $e_3$ component 
for different values of $R$.
}\label{Table_g1_R}
\begin{tabular}{|c|c|c|}
\hline
{$R$}&{$e_3$}&{$\AvE$} \\
\hline
1& 1.29$\times 10^{-1}$& 2.55$\times 10^{-2}$\\
4& 9.34$\times 10^{-2}$& 1.78$\times 10^{-2}$\\
7& 6.69$\times 10^{-2}$& 1.37$\times 10^{-2}$\\
10& 4.68$\times 10^{-2}$& 9.80$\times 10^{-3}$\\
13& 3.20$\times 10^{-2}$& 6.99$\times 10^{-3}$\\
\hline
\end{tabular}
\end{table}

\smallskip
We now turn our attention to the second major goal of this section - the numerical 
illustration of the theoretical error estimate \eqref{theorem_error_1}.
Specifically, we investigate whether the asymptotic behavior of the RHS of estimate
\eqref{theorem_error_1} with respect to $R$ is reflected in the numerical results.
We focus on the asymptotic behavior with respect to $R$ since this is the main 
parameter introduced by the VMS formulation; the asymptotic behavior with respect
to $r$ was investigated in \cite{borggaard2011artificial}, whereas the asymptotic behavior 
with respect to $h$ and $\Delta t$ is standard \cite{BS94,thomee2006galerkin}.
To investigate the asymptotic behavior with respect to $R$, we have to ensure that
$\sqrt{\alpha}\PODR$ (the only term that depends on $R$) dominates all the other
terms on the RHS of \eqref{theorem_error_1}.
To this end, we start collecting all the terms that depend on the exact solution $u$
and we include them in the generic constant $C$.
Next, we assume that the POD interpolation error in the initial condition $\| u^0 - u_r^0\|$
is negligible.
We also assume that the time-step is small enough to neglect 
$\Delta t \, \| u_{t t} \|_{L^2(L^2)}$. 
With these assumptions, the error estimate \eqref{theorem_error_1} can now be written 
as $e \leq C \left( e_1 + e_2 + e_3 \right)$, where $e$ is the VMS-POD model's average 
error, $C$ a generic constant independent of $r, R, h, \Delta t$ and $\alpha$, 
$e_1 = \MrI^{\frac{1}{2}}h^{m+1}$, $e_2=\MrI^{\frac{1}{2}}\PODr$, and 
$e_3=\sqrt{\alpha}\PODR$.
To ensure that $e_3$ dominates the other terms, we choose $r = 100$ and consider
relatively low values for $R$.
This choice for $r$, which is not optimal for practical computations, ensures, however,
that $e_3$ dominates $e_2$.
We also note that, when $h$ is small, $e_3$ dominates $e_1$ too.
Thus, to investigate the asymptotic behavior with respect to $R$ of the RHS of 
\eqref{theorem_error_1}, we fix $\alpha =5\times 10^{-3}$, vary $R$ from $1$ to $14$, 
and monitor the changes in $e_3$.
We restrict $R$ to this parameter range to ensure that $\PODR$ (and thus $e_3$) 
dominates $e_2$ and $e_1$.
Table \ref{Table_g1_R} lists the VMS-POD model's average error $e = \AvE$ and its 
$e_3$ component for different values of $R$.
We emphasize that, in this case, $e_3$ dominates the other two error components
$e_1 = 3.81\times 10^{-3}$ and $e_2 = 2.87\times 10^{-3}$.
To see whether the theoretical linear dependency predicted by the theoretical error
estimate \eqref{theorem_error_1} is recovered in the numerical results in Table 
\ref{Table_g1_R}, we utilize a linear regression analysis in Figure \ref{Figure_g1_R}.
This plot shows that the rate of convergence of $e$ with respect to $e_3$ is $0.9$, 
which is close to the theoretical value of $1$ predicted by \eqref{theorem_error_1}.
We believe that this slight discrepancy is due to the fact that the mesh-size $h = 0.01$
that we have employed in this numerical investigation is not small enough for our 
asymptotic study.

Summarizing the results above, we conclude that the theoretical error estimate in
\eqref{theorem_error_1} is recovered asymptotically (with respect to $R$) in our
numerical experiments.

\begin{figure}[h]
\centering
\caption{Linear regression of VMS-POD model's average error with respect to $e_3$. 
The convergence rate is $0.9$, which is close to the theoretical value of $1$ predicted 
by \eqref{theorem_error_1}.}
\label{Figure_g1_R}
\includegraphics[width=0.7\textwidth]{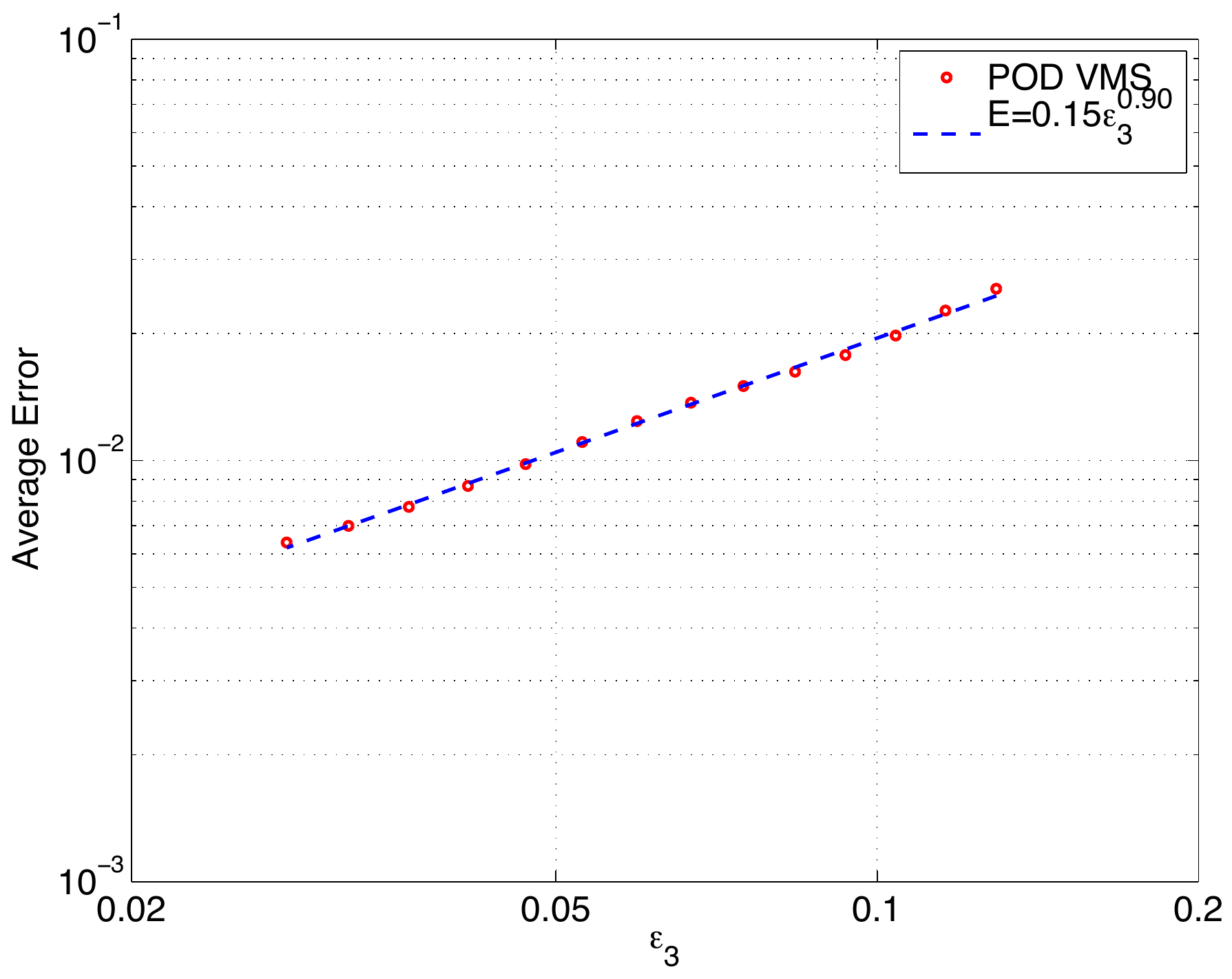} 
\end{figure}

\smallskip
Next, we use Theorem \ref{theorem_error} to provide theoretical guidance in choosing
an optimal value for the artificial viscosity coefficient $\alpha$.
The main challenge is that the theoretical error estimate \eqref{theorem_error_1} is 
{\em asymptotic} with respect to $h, \Delta t$ and $r$, while in practical computations
we are using small, yet non-negligible values for these parameters.
Furthermore, the generic constant $C$ is problem-dependent and can play a significant
role in practical computations.
Notwithstanding these hurdles, we choose a value for $\alpha$ that minimizes the RHS
of \eqref{theorem_error_1}: 
$\widetilde{\alpha}=\frac{h^{m+1}+\PODr}{2h^m+\PODr+\PODR}$.
In the derivation of this formula, we made the same assumptions as those made in the
numerical investigation of the asymptotic behavior of the VMS-POD model's error and
we again considered that \eqref{theorem_error_1} can be written as 
$e \leq C \left( e_1 + e_2 + e_3 \right)$.
We note that, if $\PODr<<\PODR$ and $h^{m}<<\PODR$, then $\widetilde{\alpha}$
becomes too small in practical computations and the VMS-POD model becomes
similar to the inaccurate POD-G model.
To circumvent this, we use in our numerical tests a ``clipping" procedure by setting
$\alpha^{*} = \max\left\{\widetilde{\alpha},\frac{h}{2}\right\}$.

Table \ref{comp_sugA} lists the VMS-POD model's average error $e = \AvE$ for the 
following values of $r, R$ and $\alpha$: $r = 20, 40$ and $60$; $R$ from $5$ to $r-5$ 
in increments of $5$; and $\alpha = 0.01 \, \alpha^{*}, \, \alpha^{*}$, and 
$100 \, \alpha^{*}$.
Note that the VMS-POD model consistently performs best for $\alpha = \alpha^{*}$.
The only two slight deviations from this rule are for $r = 60$ ($R = 20$ and $R = 30$);
we again believe that this is due to the mesh-size $h = 0.01$, which is not small 
enough for the asymptotic regime in Theorem \ref{theorem_error}.

\begin{table}[htb]
\centering
\caption{VMS-POD model's average error $e = \AvE$ for different values of $r$ and $R$,
and  $\alpha = 0.01 \, \alpha^{*}, \, \alpha^{*}$, and $100 \, \alpha^{*}$.
Note that the VMS-POD model consistently performs best for $\alpha = \alpha^{*}$.}
\label{comp_sugA}
{\small
\begin{tabular}{|c|c|cc|cc|cc|}
\hline
{$r$}&{$R$}&{$0.01\alpha^{*}$}&{$e$}&{$\alpha^{*}$}&{$e$}&{$100\alpha^{*}$}&{$e$}\\
\hline
\multirow{3}{*}{20}&5& $1.2\times10^{-3}$&$1.0\times10^{-1}$&$1.2\times10^{-1}$
&$5.8\times10^{-2}$&$1.2\times10^{1}$&$7.8\times 10^{-2}$\\
{}&10&$2.0\times10^{-3}$&$9.5\times10^{-2}$&$2.0\times10^{-1}$
& $2.4\times10^{-2}$&$2.04\times10^{1}$&$2.6\times10^{-2}$\\
{}&15&$3.3\times10^{-3}$&$8.2\times10^{-2}$&$3.3\times10^{-1}$
& $2.0\times10^{-2}$&$3.3\times10^{1}$&$2.5\times10^{-2}$\\
\hline
\multirow{5}{*}{40}&5&$6.4\times10^{-5}$&$1.09\times10^{-1}$&$6.4\times10^{-3}$&$3.0\times10^{-2}$&$6.4\times10^{-1}$ &$7.2 \times10^{-2}$\\
{}&10&$1.1\times10^{-4}$&$1.0\times10^{-1}$&$1.1\times10^{-2}$&$1.8\times10^{-2}$&$1.1\times10^{0}$ &$2.5\times10^{-2}$\\
{}&20&$4.2\times10^{-4}$&$9.7\times10^{-2}$&$4.2\times10^{-2}$&$4.4\times10^{-3}$&$4.2\times10^{0}$ &$4.1\times10^{-3}$\\
{}&30&$1.7\times10^{-3}$&$6.8\times10^{-2}$&$1.7\times10^{-1}$&$8.1\times10^{-3}$&$1.7\times10^{1}$ &$1.0\times10^{-2}$\\
{}&35&$3.0\times10^{-3}$&$4.9\times10^{-2}$&$3.0\times10^{-1}$&$2.1\times10^{-2}$&$3.0\times10^{1}$ &$2.4\times10^{-2}$\\
\hline
\multirow{3}{*}
{60}&5&{$5.0\times10^{-5}$}&{$8.7\times10^{-2}$}&$5.0\times10^{-3}$&$1.8\times10^{-2}$&{$5.0\times10^{-1}$}&{$7.0\times10^{-2}$}\\
{}&10&{$5.0\times10^{-5}$}&{$8.7\times10^{-2}$}&$5.0\times10^{-3}$&$1.3\times10^{-2}$&{$5.0\times10^{-1}$}&{$2.4\times10^{-2}$}\\
{}&20&{$5.0\times10^{-5}$}&{$8.7\times10^{-2}$}&$5.0\times10^{-3}$&$1.0\times10^{-2}$&{$5.0\times10^{-1}$}&{$3.9\times10^{-3}$}\\
{}&30&{$1.2\times10^{-4}$}&{$8.0\times10^{-2}$}&$1.2\times10^{-2}$&$4.4\times10^{-3}$&{$1.2\times10^{0}$}&{$7.4\times10^{-4}$}\\
{}&40&{$5.4\times10^{-4}$}&{$5.5\times10^{-2}$}&$5.4\times10^{-2}$&$1.2\times10^{-3}$&{$5.4\times10^{0}$}&{$2.4\times10^{-3}$}\\
{}&50&{$1.8\times10^{-3}$}&{$2.6\times10^{-2}$}&$1.8\times10^{-1}$&$1.3\times10^{-2}$&{$1.8\times10^{1}$}&{$1.4\times10^{-2}$}\\
{}&55&{$2.9\times10^{-3}$}&{$2.2\times10^{-2}$}&$2.9\times10^{-1}$&$1.1\times10^{-2}$&{$2.9\times10^{1}$}&{$1.2\times10^{-2}$}\\
\hline
\end{tabular}
}
\end{table}

\smallskip
Finally, we investigate numerically the VMS-POD model's sensitivity with respect
to changes in the diffusion coefficient $\varepsilon$.
To this end, we run the VMS-POD model \eqref{vms_pod} with the same parameters
as above ($r = 40, \, R = 20$ and $\alpha=\alpha^{*}$) for different values of the
diffusion coefficient: $\varepsilon = 10^{-2}, 10^{-4}$ and $10^{-6}$.
Table \ref{table_error_epsilons} lists the average errors for DNS, POD-G and 
VMS-POD models for different values of $\varepsilon$.
It is clear from this table that the POD-G model's average error is significantly higher
than the error of the DNS.
The VMS-POD model, however, performs well for all values of $\varepsilon$ and displays
a low sensitivity with respect to changes in the diffusion coefficient.
Thus, we provide numerical support for the theoretical estimate \eqref{theorem_error_1}, 
which is {\em uniform} with respect to $\varepsilon$.

\begin{table}[htb]
\centering
\caption{Average errors of DNS, POD-G and VMS-POD models for different values of the
diffusion coefficient $\varepsilon$.
The POD-G model performs poorly.
The VMS-POD model performs well and displays low sensitivity with respect to changes in
$\varepsilon$.}
\label{table_error_epsilons}
{\small
\begin{tabular}{|c|c|c|cc|}
\hline
\multirow{2}{*}{$\varepsilon$}&{${DNS}$}&{POD-G}&\multicolumn{2}{c|}{VMS-POD}\\
\cline{2-5}
{}&{$\frac{1}{N+1}\sum\limits_{n=0}^{N}\|u^n-u_h^N\|$}&{$\AvE$}&$\alpha$&$\AvE$\\
\hline
{$10^{-2}$}&{$1.10\times10^{-4}$}&{$1.10\times10^{-2}$}&{$4.05\times10^{-2}$}&{$4.27\times10^{-3}$}\\
{$10^{-4}$}&{$2.04\times10^{-4}$}&{$1.11\times10^{-1}$}&{$4.29\times10^{-2}$}&{$4.48\times10^{-3}$}\\
{$10^{-6}$}&{$1.88\times10^{-4}$}&{$1.17\times10^{-1}$}&{$9.65\times10^{-2}$}&$4.05\times10^{-3}$\\
{$10^{-8}$}&{$2.46\times10^{-4}$}&{$1.17\times10^{-1}$}&{$1.01\times10^{-1}$}&{$4.05\times10^{-3}$}\\
\hline
\end{tabular}
}
\end{table}

\section{Conclusions}
\label{s_conclusions}

We presented a new VMS closure modeling strategy for the numerical stabilization 
of POD-ROMs of convection-dominated equations.
The new POD-ROM, denoted VMS-POD, utilizes an artificial viscosity term to add 
numerical stabilization to the model.
Following the guiding principle of the VMS methodology, we only add artificial 
viscosity to the small resolved scales.
Thus, no artificial viscosity is used for the large resolved scales.
The POD setting represents an ideal framework for the VMS approach, since 
the POD modes are listed in descending order of their kinetic energy content.

A thorough numerical analysis for the finite element discretization of the new 
VMS-POD model was presented.
The numerical tests showed the increased numerical stability of the new VMS-POD 
model and illustrated the theoretical error estimates.
We also employed the theoretical error estimates to provide guidance in choosing
the artificial viscosity coefficient in practical computations.
We emphasize that the theoretical error estimates were uniform with respect to
$\varepsilon$, the diffusion coefficient.
The numerical tests confirmed the theoretical results:
The average error of the VMS-POD model showed a low sensitivity with respect
to changes in $\varepsilon$.

Although the new VMS-POD model targets general convection-domainted problems,
it was analyzed theoretically and tested numerically  by using the convection-dominated
convection-diffusion equations.
We chose this simplified mathematical and numerical setting as a first step in a thorough
investigation of the new VMS-POD model.
Next, we will utilize the new VMS-POD model in the numerical simulation of turbulent
flows, such as $3D$ flow past a circular cylinder \cite{wang2011two}.
We also note that, to our knowledge, this is the first time that the VMS formulation used
in \cite{layton2002connection} for the numerical stabilization of finite element discretizations
has been used in a POD setting.
We will investigate in a future study the alternative VMS formulation proposed in 
\cite{guermond1999stabilization} and compare it with the VMS-POD model that we
introduced in this report.

\bibliographystyle{plain}

\end{document}